\documentclass{amsart}
\usepackage{amssymb}
\usepackage{amsfonts}

\newtheorem{theorem}{Theorem}
\theoremstyle{plain}

\newtheorem{corollary}{Corollary}

\newtheorem{lemma}{Lemma}

\newtheorem{proposition}{Proposition}
\newtheorem{remark}{Remark}

\numberwithin{equation}{section}
\input{tcilatex}

\begin{document}
\title[Some Companions of Ostrowski type inequality]{Some Companions of
Ostrowski type inequality for functions whose second derivatives are convex
and concave with applications}
\author{M.Emin \"{O}zdemir$^{\blacklozenge }$}
\address{$^{\blacklozenge }$Atat\"{u}rk University, K.K. Education Faculty,
Department of Mathematics, Erzurum 25240, Turkey}
\email{emos@atauni.edu.tr}
\author{Merve Avci Ardic$^{\bigstar \diamondsuit }$}
\address{$^{\bigstar }$Adiyaman University, Faculty of Science and Arts,
Department of Mathematics, Adiyaman 02040, Turkey}
\email{mavci@posta.adiyaman.edu.tr}
\thanks{$^{\lozenge }$Corresponding Author}

\begin{abstract}
In this paper, we obtain some companions of Ostrowski type inequality for
absolutely continuous functions whose second derivatives absolute value are
convex and concave. Finally, we gave some applications for special means.
\end{abstract}

\maketitle

\section{\protect\bigskip introduction}

The following inequality is well known as Ostrowski's inequality in the
literature \cite{O}:

\begin{theorem}
\label{teo 1.1} Let $f:I\subset 
%TCIMACRO{\U{211d} }%
%BeginExpansion
\mathbb{R}
%EndExpansion
\rightarrow 
%TCIMACRO{\U{211d} }%
%BeginExpansion
\mathbb{R}
%EndExpansion
$ be a differentiable mapping on $I^{\circ },$ the interior of the interval $%
I,$ such that $f^{\prime }\in L[a,b],$ where $a,b\in I$ with $a<b.$ If $%
\left\vert f^{\prime }(x)\right\vert \leq M,$ then the following inequality,%
\begin{equation*}
\left\vert f(x)-\frac{1}{b-a}\int_{a}^{b}f(x)dx\right\vert \leq M\left(
b-a\right) \left[ \frac{1}{4}+\frac{\left( x-\frac{a+b}{2}\right) ^{2}}{%
\left( b-a\right) ^{2}}\right]
\end{equation*}%
holds for all $x\in \lbrack a,b].$ The constant $\frac{1}{4}$ is the best
possible in the sense that it can not be replaced by a smaller constant.
\end{theorem}

In \cite{SSO}, Set et al. proved some inequalities for $s-$concave and
concave functions via following Lemma:

\begin{lemma}
\label{lem 1.0} Let $f:I\subseteq 
%TCIMACRO{\U{211d} }%
%BeginExpansion
\mathbb{R}
%EndExpansion
\rightarrow 
%TCIMACRO{\U{211d} }%
%BeginExpansion
\mathbb{R}
%EndExpansion
$ be a twice differentiable function on $I^{\circ }$ with $f^{\prime \prime
}\in L_{1}[a,b],$ then 
\begin{eqnarray*}
&&\frac{1}{b-a}\int_{a}^{b}f(u)du-f(x)+\left( x-\frac{a+b}{2}\right)
f^{\prime }(x) \\
&=&\frac{\left( x-a\right) ^{3}}{2\left( b-a\right) }\int_{0}^{1}t^{2}f^{%
\prime \prime }\left( tx+\left( 1-t\right) a\right) dt+\frac{\left(
b-x\right) ^{3}}{2\left( b-a\right) }\int_{0}^{1}t^{2}f^{\prime \prime
}\left( tx+\left( 1-t\right) b\right) dt.
\end{eqnarray*}
\end{lemma}

\begin{theorem}
\label{teo 1.2} Let $f:I\subseteq \lbrack 0,\infty )\rightarrow 
%TCIMACRO{\U{211d} }%
%BeginExpansion
\mathbb{R}
%EndExpansion
$ be a twice differentiable function on $I^{\circ }$ such that $f^{\prime
\prime }\in L_{1}[a,b]$ where $a,b\in I$ with $a<b.$ If $\left\vert
f^{\prime \prime }\right\vert ^{q}$ is $s-$concave in the second sense on $%
[a,b]$ for some fixed $s\in (0,1],$ $p,q>1$ and $\frac{1}{p}+\frac{1}{q}=1,$
then the following inequality holds:%
\begin{eqnarray}
&&\left\vert \frac{1}{b-a}\int_{a}^{b}f(u)du-f(x)+\left( x-\frac{a+b}{2}%
\right) f^{\prime }(x)\right\vert  \label{1.1} \\
&\leq &\frac{2^{\frac{\left( s-1\right) }{q}}}{\left( 2p+1\right) ^{\frac{1}{%
p}}\left( b-a\right) }\left( \frac{\left( x-a\right) ^{3}\left\vert
f^{\prime \prime }\left( \frac{x+a}{2}\right) \right\vert +\left( b-x\right)
^{3}\left\vert f^{\prime \prime }\left( \frac{x+a}{2}\right) \right\vert }{2}%
\right)  \notag
\end{eqnarray}%
for each $x\in \left[ a,b\right] .$
\end{theorem}

\begin{corollary}
\label{co 1.1} If in (\ref{1.1}), we choose $x=\frac{a+b}{2},$ then we have 
\begin{equation*}
\left\vert \frac{1}{b-a}\int_{a}^{b}f(u)du-f\left( \frac{a+b}{2}\right)
\right\vert \leq \frac{2^{\frac{\left( s-1\right) }{q}}\left( b-a\right) ^{2}%
}{16\left( 2p+1\right) ^{\frac{1}{p}}}\left[ \left\vert f^{\prime \prime
}\left( \frac{3a+b}{4}\right) \right\vert +\left\vert f^{\prime \prime
}\left( \frac{a+3b}{4}\right) \right\vert \right] .
\end{equation*}%
For instance, if $s=1,$ then we have 
\begin{equation*}
\left\vert \frac{1}{b-a}\int_{a}^{b}f(u)du-f\left( \frac{a+b}{2}\right)
\right\vert \leq \frac{\left( b-a\right) ^{2}}{16\left( 2p+1\right) ^{\frac{1%
}{p}}}\left[ \left\vert f^{\prime \prime }\left( \frac{3a+b}{4}\right)
\right\vert +\left\vert f^{\prime \prime }\left( \frac{a+3b}{4}\right)
\right\vert \right] .
\end{equation*}
\end{corollary}

In \cite{Z}, Liu introduced some companions of an Ostrowski type inequality
for functions whose first derivative are absolutely continuous.In \cite{B},
Barnett et al. established some companions for the Ostrowski inequality and
the generalized trapezoid inequality. In \cite{H}, Alomari et al. introduced
some companions of Ostrowski inequality for functions whose first
derivatives absolute value are convex.

In this paper, we established some companions of Ostrowski type inequality
for absolutely continuous functions whose second derivatives absolute value
are convex and concave.

In order to prove our main results we need the following Lemma \cite{Z}:

\begin{lemma}
\label{lem 1.1} Let $f:[a,b]\rightarrow 
%TCIMACRO{\U{211d} }%
%BeginExpansion
\mathbb{R}
%EndExpansion
$ be such that the derivative $f^{\prime }$ is absolutely continuous on $%
[a,b].$ Then we have the equality%
\begin{eqnarray*}
&&\frac{1}{b-a}\int_{a}^{b}f(t)dt-\frac{1}{2}\left[ f(x)+f(a+b-x)\right] \\
&&+\frac{1}{2}\left( x-\frac{3a+b}{4}\right) \left[ f^{\prime }(x)-f^{\prime
}(a+b-x)\right] \\
&=&\frac{1}{2\left( b-a\right) }\left[ \int_{a}^{x}\left( t-a\right)
^{2}f^{\prime \prime }(t)dt+\int_{x}^{a+b-x}\left( t-\frac{a+b}{2}\right)
^{2}f^{\prime \prime }(t)dt\right. \\
&&\left. +\int_{a+b-x}^{b}\left( t-b\right) ^{2}f^{\prime \prime }(t)dt 
\right]
\end{eqnarray*}%
for all $x\in \left[ a,\frac{a+b}{2}\right] .$
\end{lemma}

\section{main results}

We will start with the following theorem:

\begin{theorem}
\label{teo 2.1} Let $f:[a,b]\rightarrow 
%TCIMACRO{\U{211d} }%
%BeginExpansion
\mathbb{R}
%EndExpansion
$ be a function such that $f^{\prime }$ is absolutely continuous on $[a,b],$ 
$f^{\prime \prime }\in L_{1}[a,b].$ If $\left\vert f^{\prime \prime
}\right\vert $ is convex on $[a,b],$ then we have the following inequality:%
\begin{eqnarray*}
&&\left\vert \frac{1}{b-a}\int_{a}^{b}f(t)dt-\frac{1}{2}\left[ f(x)+f(a+b-x)%
\right] \right. \\
&&\left. +\frac{1}{2}\left( x-\frac{3a+b}{4}\right) \left[ f^{\prime
}(x)-f^{\prime }(a+b-x)\right] \right\vert \\
&\leq &\frac{\left( x-a\right) ^{3}}{24(b-a)}\left[ \left\vert f^{\prime
\prime }(a)\right\vert +\left\vert f^{\prime \prime }(b)\right\vert \right]
\\
&&+\frac{6\left( x-a\right) ^{3}+\left( a+b-2x\right) ^{3}}{48(b-a)}\left[
\left\vert f^{\prime \prime }(x)\right\vert +\left\vert f^{\prime \prime
}(a+b-x)\right\vert \right]
\end{eqnarray*}%
for all $x\in \left[ a,\frac{a+b}{2}\right] .$
\end{theorem}

\begin{proof}
Using Lemma \ref{lem 1.1} and the property of the modulus we have%
\begin{eqnarray*}
&&\left\vert \frac{1}{b-a}\int_{a}^{b}f(t)dt-\frac{1}{2}\left[ f(x)+f(a+b-x)%
\right] \right. \\
&&\left. +\frac{1}{2}\left( x-\frac{3a+b}{4}\right) \left[ f^{\prime
}(x)-f^{\prime }(a+b-x)\right] \right\vert \\
&\leq &\frac{1}{2(b-a)}\left[ \int_{a}^{x}\left( t-a\right) ^{2}\left\vert
f^{\prime \prime }(t)\right\vert dt+\int_{x}^{a+b-x}\left( t-\frac{a+b}{2}%
\right) ^{2}\left\vert f^{\prime \prime }(t)\right\vert dt\right. \\
&&\left. +\int_{a+b-x}^{b}\left( t-b\right) ^{2}\left\vert f^{\prime \prime
}(t)\right\vert dt\right] .
\end{eqnarray*}%
Since $\left\vert f^{\prime \prime }\right\vert $ is convex on $[a,b],$ we
have%
\begin{equation*}
\left\vert f^{\prime \prime }(t)\right\vert \leq \frac{t-a}{x-a}\left\vert
f^{\prime \prime }(x)\right\vert +\frac{x-t}{x-a}\left\vert f^{\prime \prime
}(a)\right\vert ,\text{ \ \ \ \ }t\in \lbrack a,x];
\end{equation*}

\begin{equation*}
\left\vert f^{\prime \prime }(t)\right\vert \leq \frac{t-x}{a+b-2x}%
\left\vert f^{\prime \prime }(a+b-x)\right\vert +\frac{a+b-x-t}{a+b-2x}%
\left\vert f^{\prime \prime }(x)\right\vert ,\text{ \ \ \ \ }t\in (x,a+b-x]
\end{equation*}%
and%
\begin{equation*}
\left\vert f^{\prime \prime }(t)\right\vert \leq \frac{t-a-b+x}{x-a}%
\left\vert f^{\prime \prime }(b)\right\vert +\frac{b-t}{x-a}\left\vert
f^{\prime \prime }(a+b-x)\right\vert ,\text{ \ \ \ \ }t\in (a+b-x,b].
\end{equation*}%
Therefore we can write%
\begin{eqnarray*}
&&\left\vert \frac{1}{b-a}\int_{a}^{b}f(t)dt-\frac{1}{2}\left[ f(x)+f(a+b-x)%
\right] \right. \\
&&\left. +\frac{1}{2}\left( x-\frac{3a+b}{4}\right) \left[ f^{\prime
}(x)-f^{\prime }(a+b-x)\right] \right\vert \\
&\leq &\frac{1}{2(b-a)}\left\{ \int_{a}^{x}\left( t-a\right) ^{2}\left[ 
\frac{t-a}{x-a}\left\vert f^{\prime \prime }(x)\right\vert +\frac{x-t}{x-a}%
\left\vert f^{\prime \prime }(a)\right\vert \right] dt\right. \\
&&\left. +\int_{x}^{a+b-x}\left( t-\frac{a+b}{2}\right) ^{2}\left[ \frac{t-x%
}{a+b-2x}\left\vert f^{\prime \prime }(a+b-x)\right\vert +\frac{a+b-x-t}{%
a+b-2x}\left\vert f^{\prime \prime }(x)\right\vert \right] \right. \\
&&\left. +\int_{a+b-x}^{b}\left( t-b\right) ^{2}\left[ \frac{t-a-b+x}{x-a}%
\left\vert f^{\prime \prime }(b)\right\vert +\frac{b-t}{x-a}\left\vert
f^{\prime \prime }(a+b-x)\right\vert \right] \right\} \\
&=&\frac{1}{2(b-a)}\left\{ \frac{1}{4}\left( x-a\right) ^{3}\left\vert
f^{\prime \prime }(x)\right\vert +\frac{1}{12}\left( x-a\right)
^{3}\left\vert f^{\prime \prime }(a)\right\vert \right. \\
&&\left. +\frac{1}{24}\left( a+b-2x\right) ^{3}\left\vert f^{\prime \prime
}(a+b-x)\right\vert +\frac{1}{24}\left( a+b-2x\right) ^{3}\left\vert
f^{\prime \prime }(x)\right\vert \right. \\
&&\left. +\frac{1}{12}\left( x-a\right) ^{3}\left\vert f^{\prime \prime
}(b)\right\vert +\frac{1}{4}\left( x-a\right) ^{3}\left\vert f^{\prime
\prime }(a+b-x)\right\vert \right\} \\
&=&\frac{\left( x-a\right) ^{3}}{24(b-a)}\left[ \left\vert f^{\prime \prime
}(a)\right\vert +\left\vert f^{\prime \prime }(b)\right\vert \right] \\
&&+\frac{6\left( x-a\right) ^{3}+\left( a+b-2x\right) ^{3}}{48(b-a)}\left[
\left\vert f^{\prime \prime }(x)\right\vert +\left\vert f^{\prime \prime
}(a+b-x)\right\vert \right] ,
\end{eqnarray*}%
which is the desired result.
\end{proof}

\begin{corollary}
\label{co 2.1} Let $f$ as in Theorem \ref{teo 2.1}. Additionally, if $%
f^{\prime }(x)=f^{\prime }(a+b-x)$, we have%
\begin{eqnarray*}
&&\left\vert \frac{1}{b-a}\int_{a}^{b}f(t)dt-\frac{1}{2}\left[ f(x)+f(a+b-x)%
\right] \right\vert \\
&&\leq \frac{\left( x-a\right) ^{3}}{24(b-a)}\left[ \left\vert f^{\prime
\prime }(a)\right\vert +\left\vert f^{\prime \prime }(b)\right\vert \right]
\\
&&+\frac{6\left( x-a\right) ^{3}+\left( a+b-2x\right) ^{3}}{48(b-a)}\left[
\left\vert f^{\prime \prime }(x)\right\vert +\left\vert f^{\prime \prime
}(a+b-x)\right\vert \right] .
\end{eqnarray*}
\end{corollary}

\begin{corollary}
\label{co 2.2} In Corollary \ref{co 2.1}, if $f$ is symmetric function, $%
f(a+b-x)=f(x),$ for all $x\in \left[ a,\frac{a+b}{2}\right] $ we have 
\begin{eqnarray*}
&&\left\vert \frac{1}{b-a}\int_{a}^{b}f(t)dt-f(x)\right\vert \\
&\leq &\frac{\left( x-a\right) ^{3}}{24(b-a)}\left[ \left\vert f^{\prime
\prime }(a)\right\vert +\left\vert f^{\prime \prime }(b)\right\vert \right]
\\
&&+\frac{6\left( x-a\right) ^{3}+\left( a+b-2x\right) ^{3}}{48(b-a)}\left[
\left\vert f^{\prime \prime }(x)\right\vert +\left\vert f^{\prime \prime
}(a+b-x)\right\vert \right] ,
\end{eqnarray*}%
which is an Ostrowski type inequality.
\end{corollary}

\begin{corollary}
\label{co 2.3} In Corollary \ref{co 2.2}, if we choose $x=\frac{a+b}{2},$ we
have 
\begin{eqnarray*}
&&\left\vert \frac{1}{b-a}\int_{a}^{b}f(t)dt-f\left( \frac{a+b}{2}\right)
\right\vert \\
&\leq &\frac{\left( b-a\right) ^{2}}{192}\left[ \left\vert f^{\prime \prime
}(a)\right\vert +6\left\vert f^{\prime \prime }\left( \frac{a+b}{2}\right)
\right\vert +\left\vert f^{\prime \prime }(b)\right\vert \right] .
\end{eqnarray*}
\end{corollary}

\begin{corollary}
\label{co 2.4} In Theorem \ref{teo 2.1}, if we choose $x=\frac{3a+b}{4},$ we
have%
\begin{eqnarray*}
&&\left\vert \frac{1}{b-a}\int_{a}^{b}f(t)dt-\frac{1}{2}\left[ f\left( \frac{%
3a+b}{4}\right) +f\left( \frac{a+3b}{4}\right) \right] \right\vert \\
&\leq &\frac{\left( b-a\right) ^{2}}{1536}\left[ \left\vert f^{\prime \prime
}(a)\right\vert +7\left\vert f^{\prime \prime }\left( \frac{3a+b}{4}\right)
\right\vert +7\left\vert f^{\prime \prime }\left( \frac{a+3b}{4}\right)
\right\vert +\left\vert f^{\prime \prime }(b)\right\vert \right] .
\end{eqnarray*}
\end{corollary}

\begin{theorem}
\label{teo 2.2} Let $f:[a,b]\rightarrow 
%TCIMACRO{\U{211d} }%
%BeginExpansion
\mathbb{R}
%EndExpansion
$ be a function such that $f^{\prime }$ is absolutely continuous on $[a,b],$ 
$f^{\prime \prime }\in L_{1}[a,b].$ If $\left\vert f^{\prime \prime
}\right\vert ^{q}$ is convex on $[a,b],$ for all $x\in \left[ a,\frac{a+b}{2}%
\right] $ and $q>1,$ then we have the following inequality:%
\begin{eqnarray*}
&&\left\vert \frac{1}{b-a}\int_{a}^{b}f(t)dt-\frac{1}{2}\left[ f(x)+f(a+b-x)%
\right] \right. \\
&&\left. +\frac{1}{2}\left( x-\frac{3a+b}{4}\right) \left[ f^{\prime
}(x)-f^{\prime }(a+b-x)\right] \right\vert \\
&\leq &\frac{1}{2^{1+\frac{1}{q}}\left( b-a\right) \left( 2p+1\right) ^{%
\frac{1}{p}}}\left[ \left( x-a\right) ^{3}\left( \left\vert f^{\prime \prime
}(a)\right\vert ^{q}+\left\vert f^{\prime \prime }(x)\right\vert ^{q}\right)
^{\frac{1}{q}}\right. \\
&&\left. +\frac{\left( a+b-2x\right) ^{3}}{4}\left( \left\vert f^{\prime
\prime }(x)\right\vert ^{q}+\left\vert f^{\prime \prime }(a+b-x)\right\vert
^{q}\right) ^{\frac{1}{q}}\right. \\
&&\left. +\left( x-a\right) ^{3}\left( \left\vert f^{\prime \prime
}(a+b-x)\right\vert ^{q}+\left\vert f^{\prime \prime }(b)\right\vert
^{q}\right) ^{\frac{1}{q}}\right] ,
\end{eqnarray*}%
where $\frac{1}{p}+\frac{1}{q}=1.$
\end{theorem}

\begin{proof}
Using Lemma \ref{lem 1.1}, H\"{o}lder inequality and convexity of $%
\left\vert f^{\prime \prime }\right\vert ^{q},$ we have%
\begin{eqnarray*}
&&\left\vert \frac{1}{b-a}\int_{a}^{b}f(t)dt-\frac{1}{2}\left[ f(x)+f(a+b-x)%
\right] \right.  \\
&&\left. +\frac{1}{2}\left( x-\frac{3a+b}{4}\right) \left[ f^{\prime
}(x)-f^{\prime }(a+b-x)\right] \right\vert  \\
&\leq &\frac{1}{2\left( b-a\right) }\left\{ \left( \int_{a}^{x}\left(
t-a\right) ^{2p}dt\right) ^{\frac{1}{p}}\left( \int_{a}^{x}\left\vert
f^{\prime \prime }(t)\right\vert ^{q}dt\right) ^{\frac{1}{q}}\right.  \\
&&\left. +\left( \int_{x}^{a+b-x}\left( t-\frac{a+b}{2}\right)
^{2p}dt\right) ^{\frac{1}{p}}\left( \int_{x}^{a+b-x}\left\vert f^{\prime
\prime }(t)\right\vert ^{q}dt\right) ^{\frac{1}{q}}\right.  \\
&&\left. +\left( \int_{a+b-x}^{b}\left( t-b\right) ^{2p}dt\right) ^{\frac{1}{%
p}}\left( \int_{a+b-x}^{b}\left\vert f^{\prime \prime }(t)\right\vert
^{q}dt\right) ^{\frac{1}{q}}\right\}  \\
&\leq &\frac{1}{2\left( b-a\right) }\left\{ \left( \int_{a}^{x}\left(
t-a\right) ^{2p}dt\right) ^{\frac{1}{p}}\left( \int_{a+b-x}^{b}\left[ \frac{%
t-a}{x-a}\left\vert f^{\prime \prime }(x)\right\vert ^{q}+\frac{x-t}{x-a}%
\left\vert f^{\prime \prime }(a)\right\vert ^{q}\right] dt\right) ^{\frac{1}{%
q}}\right.  \\
&&+\left. \left( \int_{x}^{a+b-x}\left( t-\frac{a+b}{2}\right)
^{2p}dt\right) ^{\frac{1}{p}}\right.  \\
&&\left. \times \left( \int_{x}^{a+b-x}\left[ \frac{t-x}{a+b-2x}\left\vert
f^{\prime \prime }(a+b-x)\right\vert ^{q}+\frac{a+b-x-t}{a+b-2x}\left\vert
f^{\prime \prime }(x)\right\vert ^{q}\right] dt\right) ^{\frac{1}{q}}\right. 
\\
&&\left. +\left( \int_{a+b-x}^{b}\left( t-b\right) ^{2p}dt\right) ^{\frac{1}{%
p}}\left( \int_{a+b-x}^{b}\left[ \frac{t-a-b+x}{x-a}\left\vert f^{\prime
\prime }(b)\right\vert ^{q}+\frac{b-t}{x-a}\left\vert f^{\prime \prime
}(a+b-x)\right\vert ^{q}\right] dt\right) ^{\frac{1}{q}}\right\}  \\
&=&\frac{1}{2\left( b-a\right) }\left\{ \left( \frac{\left( x-a\right)
^{2p+1}}{\left( 2p+1\right) }\right) ^{\frac{1}{p}}\left( \frac{x-a}{2}%
\right) ^{\frac{1}{q}}\left( \left\vert f^{\prime \prime }(a)\right\vert
^{q}+\left\vert f^{\prime \prime }(x)\right\vert ^{q}\right) ^{\frac{1}{q}%
}\right.  \\
&&\left. +\left( \frac{2}{2p+1}\left( \frac{a+b}{2}-x\right) ^{2p+1}\right)
^{\frac{1}{p}}\left( \frac{a+b}{2}-x\right) ^{\frac{1}{q}}\left( \left\vert
f^{\prime \prime }(x)\right\vert ^{q}+\left\vert f^{\prime \prime
}(a+b-x)\right\vert ^{q}\right) ^{\frac{1}{q}}\right.  \\
&&\left. +\left( \frac{\left( x-a\right) ^{2p+1}}{\left( 2p+1\right) }%
\right) ^{\frac{1}{p}}\left( \frac{x-a}{2}\right) ^{\frac{1}{q}}\left(
\left\vert f^{\prime \prime }(a+b-x)\right\vert ^{q}+\left\vert f^{\prime
\prime }(b)\right\vert ^{q}\right) ^{\frac{1}{q}}\right\} .
\end{eqnarray*}%
When we arrange the statements above, we obtain the desired result.
\end{proof}

\begin{corollary}
\label{co 2.5} Let $f$ as in Theorem \ref{teo 2.2} . Additionally, if $%
f^{\prime }(x)=f^{\prime }(a+b-x)$, we have%
\begin{eqnarray*}
&&\left\vert \frac{1}{b-a}\int_{a}^{b}f(t)dt-\frac{1}{2}\left[ f(x)+f(a+b-x)%
\right] \right\vert \\
&\leq &\frac{1}{2^{1+\frac{1}{q}}\left( b-a\right) \left( 2p+1\right) ^{%
\frac{1}{p}}}\left[ \left( x-a\right) ^{3}\left( \left\vert f^{\prime \prime
}(a)\right\vert ^{q}+\left\vert f^{\prime \prime }(x)\right\vert ^{q}\right)
^{\frac{1}{q}}\right. \\
&&\left. +\frac{\left( a+b-2x\right) ^{3}}{4}\left( \left\vert f^{\prime
\prime }(x)\right\vert ^{q}+\left\vert f^{\prime \prime }(a+b-x)\right\vert
^{q}\right) ^{\frac{1}{q}}\right. \\
&&\left. +\left( x-a\right) ^{3}\left( \left\vert f^{\prime \prime
}(a+b-x)\right\vert ^{q}+\left\vert f^{\prime \prime }(b)\right\vert
^{q}\right) ^{\frac{1}{q}}\right] ,
\end{eqnarray*}%
for all $x\in \left[ a,\frac{a+b}{2}\right] .$
\end{corollary}

\begin{corollary}
\label{co 2.6} In Corollary \ref{co 2.5}, if $f$ is symmetric function, $%
f(a+b-x)=f(x),$ we have 
\begin{eqnarray*}
&&\left\vert \frac{1}{b-a}\int_{a}^{b}f(t)dt-f(x)\right\vert \\
&\leq &\frac{1}{2^{1+\frac{1}{q}}\left( b-a\right) \left( 2p+1\right) ^{%
\frac{1}{p}}}\left[ \left( x-a\right) ^{3}\left( \left\vert f^{\prime \prime
}(a)\right\vert ^{q}+\left\vert f^{\prime \prime }(x)\right\vert ^{q}\right)
^{\frac{1}{q}}\right. \\
&&\left. +\frac{\left( a+b-2x\right) ^{3}}{4}\left( \left\vert f^{\prime
\prime }(x)\right\vert ^{q}+\left\vert f^{\prime \prime }(a+b-x)\right\vert
^{q}\right) ^{\frac{1}{q}}\right. \\
&&\left. +\left( x-a\right) ^{3}\left( \left\vert f^{\prime \prime
}(a+b-x)\right\vert ^{q}+\left\vert f^{\prime \prime }(b)\right\vert
^{q}\right) ^{\frac{1}{q}}\right] ,
\end{eqnarray*}%
for all $x\in \left[ a,\frac{a+b}{2}\right] .$
\end{corollary}

\begin{corollary}
\label{co 2.7} In Corollary \ref{co 2.5}, if we choose $x=a$ we have%
\begin{equation*}
\left\vert \frac{1}{b-a}\int_{a}^{b}f(t)dt-\frac{f(a)+f(b)}{2}\right\vert
\leq \frac{\left( b-a\right) ^{2}}{2^{3+\frac{1}{q}}\left( 2p+1\right) ^{%
\frac{1}{p}}}\left[ \left\vert f^{\prime \prime }(a)\right\vert
^{q}+\left\vert f^{\prime \prime }(b)\right\vert ^{q}\right] ^{\frac{1}{q}}.
\end{equation*}
\end{corollary}

\begin{corollary}
\label{co 2.8} In Theorem \ref{teo 2.2} , if we choose
\end{corollary}

\begin{enumerate}
\item $x=\frac{a+b}{2},$ we have%
\begin{eqnarray*}
&&\left\vert \frac{1}{b-a}\int_{a}^{b}f(t)dt-f\left( \frac{a+b}{2}\right)
\right\vert \\
&\leq &\frac{\left( b-a\right) ^{2}}{2^{4+\frac{1}{q}}\left( 2p+1\right) ^{%
\frac{1}{p}}}\left[ \left( \left\vert f^{\prime \prime }(a)\right\vert
^{q}+\left\vert f^{\prime \prime }\left( \frac{a+b}{2}\right) \right\vert
^{q}\right) ^{\frac{1}{q}}\right. \\
&&\left. \left( \left\vert f^{\prime \prime }\left( \frac{a+b}{2}\right)
\right\vert ^{q}+\left\vert f^{\prime \prime }(b)\right\vert ^{q}\right) ^{%
\frac{1}{q}}\right] .
\end{eqnarray*}

\item $x=\frac{3a+b}{4},$ we have%
\begin{eqnarray*}
&&\left\vert \frac{1}{b-a}\int_{a}^{b}f(t)dt-\frac{1}{2}\left[ f\left( \frac{%
3a+b}{4}\right) +f\left( \frac{a+3b}{4}\right) \right] \right\vert \\
&\leq &\frac{\left( b-a\right) ^{2}}{2^{7+\frac{1}{q}}\left( 2p+1\right) ^{%
\frac{1}{p}}}\left\{ \left( \left\vert f^{\prime \prime }(a)\right\vert
^{q}+\left\vert f^{\prime \prime }\left( \frac{3a+b}{4}\right) \right\vert
^{q}\right) ^{\frac{1}{q}}\right. \\
&&\left. +2\left( \left\vert f^{\prime \prime }\left( \frac{3a+b}{4}\right)
\right\vert ^{q}+\left\vert f^{\prime \prime }\left( \frac{a+3b}{4}\right)
\right\vert ^{q}\right) ^{\frac{1}{q}}\right. \\
&&\left. +\left( \left\vert f^{\prime \prime }\left( \frac{a+3b}{4}\right)
\right\vert ^{q}+\left\vert f^{\prime \prime }(b)\right\vert ^{q}\right) ^{%
\frac{1}{q}}\right\} .
\end{eqnarray*}
\end{enumerate}

\begin{remark}
\label{rem 2.0} Using the well-known power-mean integral inequality one may
get inequalities for functions whose second derivatives absolute value are
convex. The details are omitted.
\end{remark}

We obtain the following result for concave functions.

\begin{theorem}
\label{teo 2.3} Let $f:[a,b]\rightarrow 
%TCIMACRO{\U{211d} }%
%BeginExpansion
\mathbb{R}
%EndExpansion
$ be a function such that $f^{\prime }$ is absolutely continuous on $[a,b],$ 
$f^{\prime \prime }\in L_{1}[a,b].$ If $\left\vert f^{\prime \prime
}\right\vert ^{q}$ is concave on $[a,b],$ for all $x\in \left[ a,\frac{a+b}{2%
}\right] $ and $q>1,$ then we have the following inequality:%
\begin{eqnarray*}
&&\left\vert \frac{1}{b-a}\int_{a}^{b}f(t)dt-\frac{1}{2}\left[ f(x)+f(a+b-x)%
\right] \right. \\
&&\left. +\frac{1}{2}\left( x-\frac{3a+b}{4}\right) \left[ f^{\prime
}(x)-f^{\prime }(a+b-x)\right] \right\vert \\
&\leq &\frac{1}{2\left( b-a\right) \left( 2p+1\right) ^{\frac{1}{p}}}\left[
\left( x-a\right) ^{3}\left\vert f^{\prime \prime }\left( \frac{x+a}{2}%
\right) \right\vert \right. \\
&&\left. +\frac{\left( a+b-2x\right) ^{3}}{4}\left\vert f^{\prime \prime
}\left( \frac{a+b}{2}\right) \right\vert +\left( x-a\right) ^{3}\left\vert
f^{\prime \prime }\left( \frac{a+2b-x}{2}\right) \right\vert \right]
\end{eqnarray*}%
where $\frac{1}{p}+\frac{1}{q}=1.$
\end{theorem}

\begin{proof}
From Lemma \ref{lem 1.1} and using H\"{o}lder inequality, we have%
\begin{eqnarray*}
&&\left\vert \frac{1}{b-a}\int_{a}^{b}f(t)dt-\frac{1}{2}\left[ f(x)+f(a+b-x)%
\right] \right. \\
&&\left. +\frac{1}{2}\left( x-\frac{3a+b}{4}\right) \left[ f^{\prime
}(x)-f^{\prime }(a+b-x)\right] \right\vert \\
&\leq &\frac{1}{2\left( b-a\right) }\left\{ \left( \int_{a}^{x}\left(
t-a\right) ^{2p}dt\right) ^{\frac{1}{p}}\left( \int_{a}^{x}\left\vert
f^{\prime \prime }(t)\right\vert ^{q}dt\right) ^{\frac{1}{q}}\right. \\
&&\left. +\left( \int_{x}^{a+b-x}\left( t-\frac{a+b}{2}\right)
^{2p}dt\right) ^{\frac{1}{p}}\left( \int_{x}^{a+b-x}\left\vert f^{\prime
\prime }(t)\right\vert ^{q}dt\right) ^{\frac{1}{q}}\right. \\
&&\left. +\left( \int_{a+b-x}^{b}\left( t-b\right) ^{2p}dt\right) ^{\frac{1}{%
p}}\left( \int_{a+b-x}^{b}\left\vert f^{\prime \prime }(t)\right\vert
^{q}dt\right) ^{\frac{1}{q}}\right\} .
\end{eqnarray*}%
Let us write, 
\begin{equation*}
\int_{a}^{x}\left\vert f^{\prime \prime }(t)\right\vert ^{q}dt=\left(
x-a\right) \int_{0}^{1}\left\vert f^{\prime \prime }\left( \lambda x+\left(
1-\lambda \right) a\right) \right\vert ^{q}d\lambda ,
\end{equation*}%
\begin{equation*}
\int_{x}^{a+b-x}\left\vert f^{\prime \prime }(t)\right\vert ^{q}dt=\left(
a+b-2x\right) \int_{0}^{1}\left\vert f^{\prime \prime }\left( \lambda \left(
a+b-x\right) +\left( 1-\lambda \right) x\right) \right\vert ^{q}d\lambda
\end{equation*}%
and%
\begin{equation*}
\int_{a+b-x}^{b}\left\vert f^{\prime \prime }(t)\right\vert ^{q}dt=\left(
x-a\right) \int_{0}^{1}\left\vert f^{\prime \prime }\left( \lambda b+\left(
1-\lambda \right) \left( a+b-x\right) \right) \right\vert ^{q}d\lambda .
\end{equation*}%
Since $\left\vert f^{\prime \prime }\right\vert ^{q}$ is concave on $[a,b],$
we use the Jensen integral inequality to obtain%
\begin{eqnarray*}
&&\left( x-a\right) \int_{0}^{1}\left\vert f^{\prime \prime }\left( \lambda
x+\left( 1-\lambda \right) a\right) \right\vert ^{q}d\lambda \\
&=&\left( x-a\right) \int_{0}^{1}\lambda ^{0}\left\vert f^{\prime \prime
}\left( \lambda x+\left( 1-\lambda \right) a\right) \right\vert ^{q}d\lambda
\\
&\leq &\left( x-a\right) \left( \int_{0}^{1}\lambda ^{0}d\lambda \right)
\left\vert f^{\prime \prime }\left( \frac{1}{\int_{0}^{1}\lambda
^{0}d\lambda }\int_{0}^{1}\left( \lambda x+\left( 1-\lambda \right) a\right)
d\lambda \right) \right\vert ^{q} \\
&=&\left( x-a\right) \left\vert f^{\prime \prime }\left( \frac{x+a}{2}%
\right) \right\vert ^{q}
\end{eqnarray*}%
and analogously%
\begin{equation*}
\left( a+b-2x\right) \int_{0}^{1}\left\vert f^{\prime \prime }\left( \lambda
\left( a+b-x\right) +\left( 1-\lambda \right) x\right) \right\vert
^{q}d\lambda \leq \left( a+b-2x\right) \left\vert f^{\prime \prime }\left( 
\frac{a+b}{2}\right) \right\vert ^{q},
\end{equation*}%
\begin{equation*}
\left( x-a\right) \int_{0}^{1}\left\vert f^{\prime \prime }\left( \lambda
b+\left( 1-\lambda \right) \left( a+b-x\right) \right) \right\vert
^{q}d\lambda \leq \left( x-a\right) \left\vert f^{\prime \prime }\left( 
\frac{a+2b-x}{2}\right) \right\vert ^{q}.
\end{equation*}%
Combining all above inequalities, we obtain%
\begin{eqnarray*}
&&\left\vert \frac{1}{b-a}\int_{a}^{b}f(t)dt-\frac{1}{2}\left[ f(x)+f(a+b-x)%
\right] \right. \\
&&\left. +\frac{1}{2}\left( x-\frac{3a+b}{4}\right) \left[ f^{\prime
}(x)-f^{\prime }(a+b-x)\right] \right\vert \\
&\leq &\frac{1}{2\left( b-a\right) }\left\{ \left( \frac{\left( x-a\right)
^{2p+1}}{\left( 2p+1\right) }\right) ^{\frac{1}{p}}\left( x-a\right) ^{\frac{%
1}{q}}\left\vert f^{\prime \prime }\left( \frac{x+a}{2}\right) \right\vert
\right. \\
&&\left. +\left( \frac{2}{2p+1}\left( \frac{a+b}{2}-x\right) ^{2p+1}\right)
^{\frac{1}{p}}\left( a+b-2x\right) ^{\frac{1}{q}}\left\vert f^{\prime \prime
}\left( \frac{a+b}{2}\right) \right\vert \right. \\
&&\left. +\left( \frac{\left( x-a\right) ^{2p+1}}{\left( 2p+1\right) }%
\right) ^{\frac{1}{p}}\left( x-a\right) ^{\frac{1}{q}}\left\vert f^{\prime
\prime }\left( \frac{a+2b-x}{2}\right) \right\vert \right\} \\
&\leq &\frac{1}{2\left( b-a\right) \left( 2p+1\right) ^{\frac{1}{p}}}\left[
\left( x-a\right) ^{3}\left\vert f^{\prime \prime }\left( \frac{x+a}{2}%
\right) \right\vert \right. \\
&&\left. +\frac{\left( a+b-2x\right) ^{3}}{4}\left\vert f^{\prime \prime
}\left( \frac{a+b}{2}\right) \right\vert +\left( x-a\right) ^{3}\left\vert
f^{\prime \prime }\left( \frac{a+2b-x}{2}\right) \right\vert \right]
\end{eqnarray*}%
for all $x\in \left[ a,\frac{a+b}{2}\right] $ and $\frac{1}{p}+\frac{1}{q}%
=1. $
\end{proof}

\begin{corollary}
\label{co 2.9} Let $f$ as in Theorem \ref{teo 2.3}. Additionally, if $%
f^{\prime }(x)=f^{\prime }(a+b-x)$, we have%
\begin{eqnarray*}
&&\left\vert \frac{1}{b-a}\int_{a}^{b}f(t)dt-\frac{1}{2}\left[ f(x)+f(a+b-x)%
\right] \right\vert \\
&\leq &\frac{1}{2\left( b-a\right) \left( 2p+1\right) ^{\frac{1}{p}}}\left[
\left( x-a\right) ^{3}\left( \left\vert f^{\prime \prime }\left( \frac{x+a}{2%
}\right) \right\vert +\left\vert f^{\prime \prime }\left( \frac{a+2b-x}{2}%
\right) \right\vert \right) \right. \\
&&\left. +\frac{\left( a+b-2x\right) ^{3}}{4}\left\vert f^{\prime \prime
}\left( \frac{a+b}{2}\right) \right\vert \right] ,
\end{eqnarray*}%
for all $x\in \left[ a,\frac{a+b}{2}\right] .$
\end{corollary}

\begin{corollary}
\label{co 2.10} In Corollary \ref{co 2.8}, if $f$ is symmetric function, $%
f(a+b-x)=f(x),$ we have 
\begin{eqnarray*}
&&\left\vert \frac{1}{b-a}\int_{a}^{b}f(t)dt-f(x)\right\vert \\
&\leq &\frac{1}{2\left( b-a\right) \left( 2p+1\right) ^{\frac{1}{p}}}\left[
\left( x-a\right) ^{3}\left( \left\vert f^{\prime \prime }\left( \frac{x+a}{2%
}\right) \right\vert +\left\vert f^{\prime \prime }\left( \frac{a+2b-x}{2}%
\right) \right\vert \right) \right. \\
&&\left. +\frac{\left( a+b-2x\right) ^{3}}{4}\left\vert f^{\prime \prime
}\left( \frac{a+b}{2}\right) \right\vert \right] ,
\end{eqnarray*}%
for all $x\in \left[ a,\frac{a+b}{2}\right] .$
\end{corollary}

\begin{corollary}
\label{c0 2.11} In Theorem \ref{teo 2.3}, if we choose $x=\frac{3a+b}{4},$
we have%
\begin{eqnarray*}
&&\left\vert \frac{1}{b-a}\int_{a}^{b}f(t)dt-\frac{1}{2}\left[ f\left( \frac{%
3a+b}{4}\right) +f\left( \frac{a+3b}{4}\right) \right] \right\vert \\
&\leq &\frac{\left( b-a\right) ^{2}}{128\left( 2p+1\right) ^{\frac{1}{p}}}%
\left[ \left\vert f^{\prime \prime }\left( \frac{7a+b}{8}\right) \right\vert
+2\left\vert f^{\prime \prime }\left( \frac{a+b}{2}\right) \right\vert
+\left\vert f^{\prime \prime }\left( \frac{a+7b}{8}\right) \right\vert %
\right] .
\end{eqnarray*}
\end{corollary}

\begin{remark}
\label{rem 2.1} In Theorem \ref{teo 2.3}, if we choose $x=\frac{a+b}{2},$ we
have the second inequality in Corollary \ref{co 1.1}.
\end{remark}

\section{applications for special means}

We consider the means for nonnegative real numbers $\alpha <\beta $ as
follows:

(1)\textbf{\ }The arithmetic mean:%
\begin{equation*}
A\left( \alpha ,\beta \right) =\frac{\alpha +\beta }{2},\text{ \ \ \ \ }%
\alpha ,\beta \in 
%TCIMACRO{\U{211d} }%
%BeginExpansion
\mathbb{R}
%EndExpansion
^{+}.
\end{equation*}

(2) The logarithmic mean:%
\begin{equation*}
L\left( \alpha ,\beta \right) =\frac{\beta -\alpha }{\ln \beta -\ln \alpha },%
\text{ \ \ \ \ }\alpha ,\beta \in 
%TCIMACRO{\U{211d} }%
%BeginExpansion
\mathbb{R}
%EndExpansion
^{+},\text{ }\alpha \neq \beta .
\end{equation*}

(3) The generalized logarithmic mean:%
\begin{equation*}
L_{n}\left( \alpha ,\beta \right) =\left[ \frac{\beta ^{n+1}-\alpha ^{n+1}}{%
\left( \beta -\alpha \right) \left( n+1\right) }\right] ^{\frac{1}{n}},\text{
\ \ \ \ }\alpha ,\beta \in 
%TCIMACRO{\U{211d} }%
%BeginExpansion
\mathbb{R}
%EndExpansion
^{+},\text{ }\alpha \neq \beta ,\text{ }n\in 
%TCIMACRO{\U{2124} }%
%BeginExpansion
\mathbb{Z}
%EndExpansion
\backslash \left\{ -1,0\right\} .
\end{equation*}

(4) The identric mean:%
\begin{equation*}
I\left( \alpha ,\beta \right) =\left\{ 
\begin{array}{ccc}
\frac{1}{e}\left( \frac{\beta ^{\beta }}{\alpha ^{\alpha }}\right) ^{\frac{1%
}{\beta -\alpha }}, &  & \alpha \neq \beta \\ 
&  &  \\ 
\alpha , &  & \text{\ }\alpha =\beta%
\end{array}%
\right. \alpha ,\beta \in 
%TCIMACRO{\U{211d} }%
%BeginExpansion
\mathbb{R}
%EndExpansion
^{+}.
\end{equation*}

Now, using the results of Section 2, we give some applications to special
means of real numbers.

\begin{proposition}
\label{prop 3.1} Let $a,b\in 
%TCIMACRO{\U{211d} }%
%BeginExpansion
\mathbb{R}
%EndExpansion
^{+},$ $a<b.$ Then we have%
\begin{equation*}
\left\vert L^{-1}\left( a,b\right) -A\left( \frac{4}{3a+b},\frac{4}{a+3b}%
\right) \right\vert \leq \frac{\left( b-a\right) ^{2}}{768}\left[ \frac{%
a^{3}+b^{3}}{a^{3}b^{3}}+448\left( \frac{1}{\left( 3a+b\right) ^{3}}+\frac{1%
}{\left( a+3b\right) ^{3}}\right) \right] .
\end{equation*}
\end{proposition}

\begin{proof}
The assertion follows from Corollary \ref{co 2.4} applied to the convex
mapping $f:\left[ a,b\right] \rightarrow 
%TCIMACRO{\U{211d} }%
%BeginExpansion
\mathbb{R}
%EndExpansion
,$ $f(x)=\frac{1}{x}.$
\end{proof}

\begin{proposition}
\label{prop 3.2} Let $a,b\in 
%TCIMACRO{\U{211d} }%
%BeginExpansion
\mathbb{R}
%EndExpansion
^{+},$ $a<b$ and $n\in 
%TCIMACRO{\U{2124} }%
%BeginExpansion
\mathbb{Z}
%EndExpansion
,$ $\left\vert n\right\vert \geq 2.$ Then we have%
\begin{equation*}
\left\vert L_{n}^{n}\left( a,b\right) -A^{n}\left( a,b\right) \right\vert
\leq \frac{n\left( n-1\right) \left( b-a\right) ^{2}}{192}\left[
a^{n-2}+6A^{n-2}\left( a,b\right) +b^{n-2}\right] .
\end{equation*}
\end{proposition}

\begin{proof}
The assertion follows from Corollary \ref{co 2.3} applied to the convex
mapping  $f:\left[ a,b\right] \rightarrow 
%TCIMACRO{\U{211d} }%
%BeginExpansion
\mathbb{R}
%EndExpansion
,$ $f(x)=x^{n}.$
\end{proof}

\begin{proposition}
\label{prop 3.3} Let $a,b\in 
%TCIMACRO{\U{211d} }%
%BeginExpansion
\mathbb{R}
%EndExpansion
^{+},$ $a<b.$ Then we have%
\begin{equation*}
\left\vert A\left( \ln a,\ln b\right) -\ln I\right\vert \leq \frac{\left(
b-a\right) ^{2}}{2^{3+\frac{1}{q}}\left( 2p+1\right) ^{\frac{1}{p}}}\left[ 
\frac{1}{a^{2q}}+\frac{1}{b^{2q}}\right] .
\end{equation*}
\end{proposition}

\begin{proof}
The assertion follows from Corollary \ref{co 2.7} applied to the convex
mapping $f:\left[ a,b\right] \rightarrow \lbrack 0,\infty ),$ $f(x)=-\ln x.$
\end{proof}

\begin{proposition}
\label{prop 3.4}  Let $a,b\in 
%TCIMACRO{\U{211d} }%
%BeginExpansion
\mathbb{R}
%EndExpansion
^{+},$ $a<b.$ Then for all $q>1,$ we have%
\begin{eqnarray*}
\left\vert L^{-1}\left( a,b\right) -A\left( \frac{4}{3a+b},\frac{4}{a+3b}%
\right) \right\vert  &\leq &\frac{\left( b-a\right) ^{2}}{2^{7+\frac{1}{q}%
}\left( 2p+1\right) ^{\frac{1}{p}}}\left\{ \left( \frac{2}{a^{3q}}+\frac{128%
}{\left( 3a+b\right) ^{3q}}\right) ^{\frac{1}{q}}\right.  \\
&&\left. +2\left( \frac{128}{\left( 3a+b\right) ^{3q}}+\frac{128}{\left(
a+3b\right) ^{3q}}\right) ^{\frac{1}{q}}+\left( \frac{128}{\left(
a+3b\right) ^{3q}}+\frac{2}{b^{3q}}\right) ^{\frac{1}{q}}\right\} .
\end{eqnarray*}
\end{proposition}

\begin{proof}
The assertion follows from second inequality in Corollary \ref{co 2.8}
applied to the convex mapping $f:\left[ a,b\right] \rightarrow 
%TCIMACRO{\U{211d} }%
%BeginExpansion
\mathbb{R}
%EndExpansion
,$ $f(x)=\frac{1}{x}.$
\end{proof}

\end{document}